\documentclass[11pt,english]{article}
\usepackage{amsfonts,amsmath,amsthm,amscd,amssymb,latexsym}

    \usepackage[T2A]{fontenc}
    \usepackage[cp1251]{inputenc}
    \usepackage{graphicx}

\begin{document}

\title{\bf On the lower bounds of
$p$-modulus of families of paths and a finite connectedness}

\author{{\bf Evgeny Sevost'yanov}, {\bf Zarina Kovba}, \\ {\bf Heorhii
Nosal,} {\bf Nataliya Ilkevych}}

\theoremstyle{plain}
\newtheorem{theorem}{Theorem}[section]
\newtheorem{lemma}{Lemma}[section]
\newtheorem{proposition}{Proposition}[section]
\newtheorem{corollary}{Corollary}[section]
\theoremstyle{definition}

\newtheorem{example}{Example}[section]
\newtheorem{remark}{Remark}[section]
\newcommand{\keywords}{\textbf{Key words.  }\medskip}
\newcommand{\subjclass}{\textbf{MSC 2000. }\medskip}
\renewcommand{\abstract}{\textbf{Abstract.}\medskip}
\numberwithin{equation}{section}

\setcounter{section}{0}
\renewcommand{\thesection}{\arabic{section}}
\newcounter{unDef}[section]
\def\theunDef{\thesection.\arabic{unDef}}
\newenvironment{definition}{\refstepcounter{unDef}\trivlist
\item[\hskip \labelsep{\bf Определение \theunDef.}]}%
{\endtrivlist}

\maketitle

\begin{abstract}
We study the problem of the lower bounds of the modulus of families
of paths of order $p,$ $p>n-1,$ and their connection with the
geometry of domains containing the specified families. Among other
things, we have proved an analogue of N\"akki's theorem on the
positivity of the $p$-modulus of families of paths joining a pair of
continua in the given domain. The geometry of domains with a
strongly accessible boundary in the sense of the $p$-modulus of
families of paths was also studied. We show that domains with a
$p$-strongly accessible boundary with respect to a $p$-modulus,
$p>n-1,$ are are finitely connected at their boundary. The mentioned
result generalizes N\"akki's result, which was proved for uniform
domains in the case of a conformal modulus.
\end{abstract}



\section{Introduction}

As is known, the lower estimates of the modulus of families of paths
are of significant importance in the study of mappings with finite
distortion. In particular, such estimates are used in the study of
the boundary behavior of mappings, see, e.g.,
\cite[Theorem~13.1]{MRSY}, \cite[Theorem~3.6]{Na$_2$} and
\cite[Theorem~17.15]{Va}. Let us give some definitions. A Borel
function $\rho:{\Bbb R}^n\,\rightarrow [0,\infty] $ is called {\it
admissible} for family $\Gamma$ of paths $\gamma$ in ${\Bbb R}^n,$
if
\begin{equation}\label{eq1.4}
\int\limits_{\gamma}\rho (x)\, |dx|\geqslant 1
\end{equation}
holds for any (locally rectifiable) path $\gamma \in \Gamma.$ %
In this case, we write: $\rho \in {\rm adm} \,\Gamma.$ Set
\begin{equation}\label{eq1A} M(\Gamma)=\inf\limits_{\rho \in
\,{\rm adm}\,\Gamma} \int\limits_{{\Bbb R}^n} \rho^n (x)\,dm(x)\,.
\end{equation}
The quantity $M(\Gamma)$ is called the {\it modulus of the family of
paths $\Gamma.$} Let $E_0, E_1\subset \overline{{\Bbb R}^n},$
$\overline{{\Bbb R}^n}={\Bbb R}^n\cup\{\infty\},$ and let $D$ be a
domain in $\overline{{\Bbb R}^n}, $ $n\geqslant 2.$ Denote by
$\Gamma (E_0, E_1, D)$ the family of all paths joining $E_0$ and
$E_1$ in $D$. The following result was established by R.~N\"{a}kki
(see~\cite[Lemma~1.15]{Na$_2$}).

\medskip
{\bf Theorem~A. (The positivity of the modulus of families of paths
joining a pair of continua)}. {\it Let $D$ be a domain in ${\Bbb
R}^n,$ $n\geqslant 2.$ If $A$ and $A^{\,*}$ are (nondegenerate)
continua in $ D,$ then $M(\Gamma(A, A^{\,*}, D))>0.$}

\medskip
Below, we will extend N\"{a}kki's result to the case of a modulus of
order $p\geqslant 1.$ Given $p\geqslant 1,$ we define {\it
$p$-modulus} of the family $\Gamma $ by the equality
\begin{equation}\label{eq1.3gl0}
M_p(\Gamma)=\inf\limits_{\rho \in \,{\rm adm}\,\Gamma}
\int\limits_{{\Bbb R}^n} \rho^p (x)\,dm(x)\,.
\end{equation}
The following theorem holds.

\medskip
\begin{theorem}\label{th2}
{\it Let $D$ be a domain in ${\Bbb R}^n,$ $n\geqslant 2,$ and let
$p>n-1.$ If $A$ and $A^{\,*}$ are (nondegenerate) continua in $D,$
then $M_p(\Gamma(A, A^{\,*}, D))>0.$}
\end{theorem}

\medskip
In this manuscript, separate studies are devoted to the geometry of
so-called domains with strongly accessible boundaries. These
boundaries are significantly used in the study of the boundary
behavior of mappings. In particular, quasiconformal mappings of
domains locally connected at their boundary onto domains which have
a strongly accessible boundary have a continuous boundary extension
(see, e.g., \cite[Theorem~17.15]{Va}, \cite[Theorem~3.6]{Na$_2$}).

\medskip Let $p\geqslant 1$ and let $D$ be a domain in
${\Bbb R}^n,$ $n\geqslant 2.$ We say that $\partial D$ is {\it
strongly accessible at the point $x_0\in
\partial D$ with respect to $p$-modulus,} if for any neighborhood $U$
of $x_0$ there is a compact set $E\subset D,$ a neighborhood
$V\subset U$ of $x_0$ and a number $\delta>0$ such that
\begin{equation}
\label{eq1.3_a} M_p(\Gamma(E,F, D))\geqslant\delta
\end{equation}
for any continuum $F$ in~$D$ with $F\cap \partial U\ne\varnothing\ne
F\cap \partial V.$ The boundary $\partial D$ of a domain $D$ in
${\Bbb R}^n,$ $n\geqslant 2,$ is {\it strongly accessible at the
point $x_0\in
\partial D$,} if it is strongly accessible with respect
to $n$-modulus at this point.

\medskip The following statement was proved in \cite[Lemma~3.8.1]{Sev$_1$}
for the case $p=n;$ see also the classical N\"{a}kki's result
in~\cite[Theorem~1.16]{Na$_2$} for the conformal modulus. (We need
to clarify that N\"{a}kki used a slightly different boundary
condition in the above paper. It is slightly different from what we
call strong accessibility).

\medskip
\begin{theorem}\label{th3}{\it\, Let $D$ be a domain in ${\Bbb R}^n,$
$n\geqslant 2,$ let $p>n-1$ and let $x_0\in
\partial D,$ $x_0\ne \infty.$ Assume that, for any neighborhood
$U$ of $x_0$ there is a compact set $E\subset D,$ a neighborhood
$V\subset U$ of $x_0$ and a number $\delta>0$ such that the relation
\begin{equation}\label{eq17***}
M_p(\Gamma(E,F, D))\geqslant \delta
\end{equation}
holds for any continuum $F$ in $D$ with $F\cap \partial
U\ne\varnothing\ne F\cap \partial V.$ Then the following holds: for
any neighborhood $U$ of $x_0$ and any continuum $E^{\,*}\subset D$
there is a neighborhood $V\subset U$ of $x_0$ and a number
$\delta^{\,*}>0$ such that
\begin{equation}\label{eq20***} M_p(\Gamma(E^{\,*}, F,
D))\geqslant \delta^{\,*}
\end{equation}
for any continuum $F$ in $D$ with $F\cap \partial U\ne\varnothing\ne
F\cap \partial V.$ }
\end{theorem}

\medskip
Observe that, the sa-called uniform domains have strongly accessible
boundaries (see, e.g., \cite[Section~2.4]{NP}). Given $p\geqslant
1,$ we say that a domain $D\subset {\Bbb R}^n,$ $n\geqslant 2,$ is
uniform with a respect $p$-modulus, if for any $r>0$ there exists
$\delta>0$ such that
\begin{equation}\label{eq17***A}
M_p(\Gamma(F^{\,*},F, D))\geqslant \delta
\end{equation}
holds for any connected sets $F, F^*\subset D$ such that
$h(F)\geqslant r$ and $h(F^{\,*})\geqslant r.$ Here $h$ denotes the
{\it spherical distance in $\overline{{\Bbb R}^n}$:}
\begin{equation*}\label{eqA17}h(x,\infty)=\frac{1}{\sqrt{1+{|x|}^2}},
\ \ h(x,y)=\frac{|x-y|}{\sqrt{1+{|x|}^2} \sqrt{1+{|y|}^2}}\,, \ \
x\ne \infty\ne y\,.\end{equation*}
Given $E\,\subset\,\overline{{\Bbb R}^n}$ we set
\begin{equation}\label{chdiam}
h(E)=\sup\limits_{x\,,y\,\in\,E}\,h(x,y)\,.
\end{equation}
A domain $D$ is {\it uniform}, if it is uniform with a respect to
$n$-modulus. R.~N\"{a}kki proved that uniform domains are finitely
connected at their boundary, see~\cite[Theorem~6.4]{Na$_1$}.

\medskip
{\bf Theorem~B. (On the finite connectedness of uniform domains on
the boundary).} {\it Let $D$ be a domain in ${\Bbb R}^n,$
$n\geqslant 2,$ which has a strongly accessible boundary at $x_0\in
\partial D.$ Then $D$ is finitely connected at the point $x_0,$
in other words, for any neighborhood $U$ of $x_0$ there exists a
neighborhood $V\subset U$ of the same point such that $V\cap D$ has
a finite number of components.}

\medskip
We extend N\"{a}kki's result to a wider class of domains that have a
strongly accessible boundary with respect to $p$-modulus, $p>n-1.$
In the form in which this statement is given here, it is new even
for $p=n.$

\medskip
\begin{theorem}\label{th5}
{\it Let $D$ be a domain in ${\Bbb R}^n,$ $n\geqslant 2,$ which has
strongly accessible boundary at the point $x_0\in
\partial D\setminus\{\infty\}$ with respect to $p$-modulus,
$p> n-1.$ Then $D$ is finitely connected at the point $x_0,$ in
other words, for any neighborhood $U$ of $x_0$ there exists a
neighborhood $V\subset U$ of the same point such that $V\cap D$ has
a finite number of components.}
\end{theorem}

\section{On $p$-modulus of families of paths joining two continua}

Let $x_0\in\overline{D},$ $x_0\ne\infty,$
$$
B(x_0, r)=\{x\in {\Bbb R}^n: |x-x_0|<r\}\,, \quad {\Bbb B}^n:=B(0,
1)\,,$$
\begin{equation}\label{eq1ED}S(x_0,r) = \{ x\,\in\,{\Bbb R}^n : |x-x_0|=r\}\,,
\end{equation}
$${\Bbb B}^{n}:=B(0, 1)\,,$$
\begin{equation}\label{eq1**A} A=A(x_0, r_1, r_2)=\{ x\,\in\,{\Bbb R}^n :
r_1<|x-x_0|<r_2\}\,.
\end{equation}
In what follows, $\Omega_n$ denotes the volume of the unit ball
${\Bbb B}^n$ in ${\Bbb R}^n.$
Let $E_0,$ $E_1$ be sets in $D\subset {\Bbb R}^n$. The following
estimate holds (see \cite[Theorem~4]{Car94}).

\medskip
\begin{proposition}\label{pr1}
{\it Let $A(0, a, b)=\{a<|x|<b\}$ be a ring containing in $D\subset
{\Bbb R}^n$ such that $S(0, r)$ intersects $E_0$ and $E_1$ for any
$r\in (a,b)$ where $E_0\cap E_1=\varnothing.$ Then for any
$p\in(n-1, n)$
\begin{equation*}
M_p(\Gamma(E_0, E_1, D))\geqslant
\frac{2^nb_{n,p}}{n-p}(b^{n-p}-a^{n-p})\,,
\end{equation*}
where $b_{n,p}$ is a constant depending only $n$ and $p$.}
\end{proposition}

\medskip
Assume that, a mapping $f$ has partial derivatives almost everywhere
in $D.$ In this case, we set
$$l\left(f^{\,\prime}(x)\right)\,=\,\min\limits_{h\in {\Bbb R}^n
\backslash \{0\}} \frac {|f^{\,\prime}(x)h|}{|h|}\,,$$
\begin{equation}\label{eq5_a}
\Vert f^{\,\prime}(x)\Vert\,=\,\max\limits_{h\in {\Bbb R}^n
\backslash \{0\}} \frac {|f^{\,\prime}(x)h|}{|h|}\,,
\end{equation}
$$J(x,f)=\det
f^{\,\prime}(x)\,.$$
and define for any $x\in D$ and fixed $p,q$, $p,q\geqslant 1$
\begin{equation}\label{eq0.1.1A}
K_{I, q}(x,f)\quad =\quad\left\{
\begin{array}{rr}
\frac{|J(x,f)|}{{l\left(f^{\,\prime}(x)\right)}^q}, & J(x,f)\ne 0,\\
1,  &  f^{\,\prime}(x)=0, \\
\infty, & {\rm otherwise}
\end{array}
\right.\,,\end{equation}
$$K_{O, p}(x,f)\quad =\quad \left\{
\begin{array}{rr}
\frac{\Vert f^\prime(x)\Vert^p}{|J(x,f)|}, & J(x,f)\ne 0,\\
1,  &  f^{\,\prime}(x)=0, \\
\infty, & {\rm otherwise}
\end{array}
\right.\,.$$
The functions $K_{I, q}$ and $K_{O, p}$ in~(\ref{eq0.1.1A}) are
called {\it inner dilatation of the order $q$} and {\it outher
dilatation of the order $p,$} respectively.

\medskip
Recall that a~mapping $f\colon X\!\rightarrow\! Y$ between measure
spaces $(X,\Sigma,\mu)$
and~$\left(Y,\Sigma^{\,\prime},\mu^{\,\prime}\right)$ has the
(Luzin) $N$-{\it\,property} if the condition $\mu(S)=0$ implies that
$\mu^{\,\prime}(f(S))=0$. A~mapping $f\colon X\!\rightarrow\! Y$ has
the (Luzin) $N^{\,-1}$-{\it\,property} if the condition
$\mu^{\,\prime}(S^{\,\prime})=0$ implies that $\mu(f^{\,-1}(S))=0.$
The following statement holds.

\medskip
\begin{lemma}\label{lem1} {\it\,
Let $f:D\rightarrow {\Bbb R}^n,$ $n\geqslant 2,$ be a homeomorphism
that is differentiable almost everywhere, has $( N)$- and
$(N^{\,-1})$-properties of Luzin. Suppose that $f^{\,-1}\in W_{\rm
loc}^{1, p}(f(D))$ for some $p\geqslant 1.$ Then the relation
\begin{equation}\label{eq5*AA}
M_p(f(\Gamma))\leqslant \int\limits_D K_{I, p}(x, f)\cdot\rho^p(x)\,
dm(x)
\end{equation}
holds for each family of paths $\ Gamma$ in $D$ and any function
$\rho\in {\rm adm\,}\Gamma.$ }
\end{lemma}
\begin{proof}
Since $f$ is differentiable almost everywhere and has $(N)$- and
$(N^{\,-1})$-Luzin properties, $f$ has finite metric distortion (see
\cite[Corollary~8.1]{MRSY}). Furthermore, since $f^{\,-1}\in W_{\rm
loc}^{1, p}(f(D)),$ by Fuglede's lemma $f^{\,-1}$ is absolutely
continuous on $p$-almost all paths (see, e.g., Theorem~28.2 in
\cite{Va}). In this case, the desired statement follows
by~\cite[Theorem~1.1]{SalSev}.
\end{proof}

\medskip
The following statement generalizes the corresponding N\"{a}kki's
result for a conformal modulus (see~\cite[Lemma~1.14]{Na$_2$}).

\medskip
\begin{lemma}\label{lem2}
{\it\, Let $p\geqslant 1,$ and let $D$ be a domain in ${\Bbb R}^n,$
$n\geqslant 2.$ Let $a_0, a^{\,\prime}_0\in D$ and let $D_0$ be a
subdomain of $D$ which contains $a_0$ and $a^{\,\prime}_0.$ Then
there exists a homeomorphism $f:\overline{D}\rightarrow
\overline{D}$ such that $M_p(f(\Gamma))\leqslant K_{n, p}\cdot
M_p(\Gamma)$ for any family $\Gamma$ in $D$ and some constant $K_{n,
p}>0$ which depends only on $n$ and $p,$ while
$f(a_0)=a^{\,\prime}_0$ and $f(x)=x$ for any $x\in
\overline{D}\setminus D_0.$}
\end{lemma}

\medskip
\begin{proof}
Let us to apply the scheme of the proof of Lemma~1.14 in
\cite{Na$_2$}. Let $L_0$ ne a polygon line which sequently joins the
points $a_0, a_1,\ldots, a_q, a^{\,\prime}_0$ in $D_0.$ Let us first
construct a mapping $f_0:\overline{D}\rightarrow \overline{D}$ such
that $M_p(f_0(\Gamma))\leqslant K^{0}_{n, p}\cdot M_p(\Gamma )$ for
each family of paths $\Gamma$ in $D$ and for some constant
$K^{0}_{n, p}>0,$ depending only on $n$ and $p,$ and $f_0(a_0)= a_1$
and $f_0(x)=x$ for all $x\in \overline{D}\setminus D_0.$

\medskip
Let $0<d_0<d(L_0, \partial D_0)/n^{1/2}.$ Without loss of
generality, we may assume that $a_0=d_0e_n$ and $a_1=d_1e_n,$
$0<d_0<d_1.$ In in the opposite case, we apply an additional
similarity transformation of the form $\varphi(x)=a+Ax,$ where $a\in
{\Bbb R}^n$ and $A$ is an orthogonal transformation such that
$\varphi(a_0)=d_0e_n$ and $\varphi(a_1)=d_1e_n,$ $0<d_0<d_1.$ For
this transformation, we obtain that $\varphi^{\,\prime}(x)=A,$ $\det
A=1$ and $l(\varphi^{\,\prime}(x))=1.$ Then $K_{I, p}(x,
\varphi)=\frac{\det A}{\lambda^p_1}=1 $ and by Lemma~\ref{lem1}
\begin{equation}\label{eq1}
M_p(\varphi(\Gamma))\leqslant \frac{\det A}{\lambda^p_1}
M_p(\Gamma)=M_p(\Gamma)\,.
\end{equation}
Due to the relation~(\ref{eq1}), it is sufficiently to construct the
corresponding mapping $f_0$ in $\varphi(D)$ and then consider the
composition $f_0\circ \varphi.$

\medskip
Set
$$C=\{x\in {\Bbb R}^n: 0\leqslant |x-x_ne_n|\leqslant d_0, \qquad 0\leqslant
x_n\leqslant d_0+d_1\}\,,$$
$$C^{\,\prime}=\{x\in {\Bbb R}^n: 0\leqslant |x-x_ne_n|\leqslant d_0-x_n,
\qquad 0\leqslant x_n\leqslant d_0\}\,,$$
$$C^{\,\prime
\prime}=\{x\in {\Bbb R}^n: 0\leqslant |x-x_ne_n|\leqslant
d_0(x_n-d_0)/d_1, \qquad d_0\leqslant x_n\leqslant d_0+d_1\}\,,$$
$$f_0(x)=\begin{cases}x\,,& x\in \overline{D}\setminus C\,,\\
x+\frac{d_1-d_0}{d_0}(d_0-|x-x_ne_n|)e_n,& x\in C\setminus
(C^{\,\prime}\cup C^{\,\prime\prime})\,, \\
x+\frac{d_1-d_0}{d_0}x_ne_n\,, &x\in C^{\,\prime}\,, \\
x+\frac{d_1-d_0}{d_1}(d_0+d_1-x_n)e_n\,, &x\in C^{\,\prime\prime}\,.
\end{cases}$$
Under the proof of Lemma~1.14 in~\cite{Na$_2$} it is proved that
$f_0$ quasiconformally maps $D$ onto itself with
$f_0(\overline{D})=\overline{D},$ while $f_0(x)=x$ for $x\in
\overline{D}\setminus D_0.$ It remains to establish the relation
$M_p(f_0(\Gamma))\leqslant K^{0}_{n, p}\cdot M_p(\Gamma)$ for any
family of paths $\Gamma$ in $D$  and some constant $K^{0}_{n, p}>0.$

\medskip
Indeed, since $f_0$ is quasiconformal, $f_0\in ACL(D).$ Set
$$\widetilde{f}_1(x)=x\,,\qquad
\widetilde{f}_2(x)=x+\frac{d_1-d_0}{d_0}(d_0-|x-x_ne_n|)e_n\,,$$
$$\widetilde{f}_3(x)=x+\frac{d_1-d_0}{d_0}x_ne_n,\qquad
\widetilde{f}_4(x)=x+\frac{d_1-d_0}{d_1}(d_0+d_1-x_n)e_n\,.$$
By the direct calculations we may show that $\Vert
\widetilde{f}_3^{\,\prime}(x)\Vert=\frac{d_1}{d_0},$ $J(x,
\widetilde{f}_3)=\frac{d_1}{d_0},$ $\Vert
\widetilde{f}_4^{\,\prime}(x)\Vert=1,$ $J(x,
\widetilde{f}_4)=\frac{d_0}{d_1},$
$l(\widetilde{f}_3^{\,\prime}(x))=1$ і
$l(\widetilde{f}_4^{\,\prime}(x))=\frac{d_0}{d_1}.$ Besides that,
$\widetilde{f}_2$ is Lipschitz mapping with a Lipschitz constant
$C=\frac{d_1}{d_0}.$ It follows from that $\Vert
\widetilde{f}_2^{\,\prime}(x)\Vert\leqslant \frac{d_1}{d_0}.$ We
also have $J(x, \widetilde{f}_2)=1.$

\medskip
Thus, $f_0\in W_{\rm loc}^{1, p}(D)$ for any $p\geqslant 1,$ because
$f_0\in ACL$ and $$\Vert
f_0^{\,\prime}(x)\Vert\leqslant\max\left\{1, \frac{d_1}{d_0},
\frac{d_0}{d_1}\right\}=\frac{d_1}{d_0}\in L^1_{\rm loc}(D)\,.$$

\medskip
Finally, $K_{I, p}(x, f)\leqslant K^{p-1}_{O, \alpha}(x, f),$ where
$$K_{O, \alpha}(x,f)\quad =\quad\left\{
\begin{array}{rr}
\frac{\Vert f^{\,\prime}(x)\Vert^{\alpha}}{|J(x,f)|}, & J(x,f)\ne 0,\\
1,  &  f^{\,\prime}(x)=0, \\
\infty, & {\rm в\,\,інших\,\,випадках}
\end{array}
\right.$$
and $\alpha=\frac{p(n-1)}{p-1}$ (see relation~(10) in \cite{Sal}).
Thus
$$K_{I, p}(x, \widetilde{f}_2)\leqslant
\left(\frac{d_1}{d_0}\right)^{\frac{p(n-1)}{p-1}\cdot
(p-1)}=\left(\frac{d_1}{d_0}\right)^{p(n-1)}\,.$$ Therefore
$$K_{I, p}(x,f_0)\leqslant
\max\left\{\left(\frac{d_1}{d_0}\right)^{p(n-1)},\frac{d_1}{d_0},\left(
\frac{d_1}{d_0}\right)^p
\right\}=\left(\frac{d_1}{d_0}\right)^{p(n-1)}:=K^{0}_{n, p}\,.$$
Obviously, the mapping $f_0$ is differentiable almost everywhere and
has $(N)$- and $(N)^{\,-1}$-properties of Luzin. Thus, by
Lemma~\ref{lem1} the mapping $f_0$ satisfies the relation
$M_p(f_0(\Gamma))\leqslant K^0_{n, p}\cdot M_p(\Gamma)$ with a
constant $K^0_{n, p}$ mentioned above.

\medskip
Similarly, for any $i=1,2,\ldots, q$ there exists a homeomorphism
$f_i:\overline{D}\rightarrow \overline{D}$ such that
$M_p(f_i(\Gamma))\leqslant K^{i}_{n, p}\cdot M_p(\Gamma)$ for any
family $\Gamma$ in $D$ and such a constant $K^{i}_{n, p}>0,$
depending only on $i,$ $n$ and $p,$ while $f(a_i)=a_{i+1}$ and
$f_i(x)=x$ for any $x\in \overline{D}\setminus D_0.$ Now we put
$f:=f_q\circ\ldots f_1\circ f_0$ which is the desired mapping.
\end{proof}

\medskip
{\it Proof of Theorem~\ref{th2}.} Since the statement of the theorem
was established for $p=n$ earlier (\cite[Lemma~1.15]{Na$_2$}), we
need to prove it for $p>n-1,$ $p\ne n.$ In the whole, we use the
scheme of the proof of~\cite[Lemma~1.15]{Na$_2$}. We may assume that
$A\cap A_*=\varnothing,$ because in this case the statement of the
Theorem is obvious. Since $A$ and $A^{\,*}$ are compact sets in
${\Bbb R}^n,$ there are $a_0\in A, a^{\,*}\in A^{\,*}$ such that
$|a_0-a^{\,*}|=d(A, A^{\,*}).$ Set
\begin{equation}\label{eq2}
2r=\min\{d(A, A^{\,*}),\, d(A^{\,*}, \partial D),\,\max\limits_{x\in
A^{\,*}}d(x, a^{\,*})\}\,.
\end{equation}
Let us choose a point $a_0^{\,\prime}\in D\setminus A^{\,*}$ such
that $d(a_0^{\,\prime}, a^{\,*})=r,$ and a subdomain $D_0$ of $D$
such that $a_0, a_0^{\,\prime}\in D_0,$ $A\cap ({\Bbb R}^n\setminus
D_0)\ne \varnothing$ and $A^{\,*}\cap D_0=\varnothing.$ Now, let $f$
be a mapping from Lemma~\ref{lem2}. Then $f(A^{\,*})=A^{\,*}.$ Let
us show that $S(a^{\,*}, t)\cap f(A)\ne\varnothing$ and $S(a^{\,*},
t)\cap A^{\,*}\ne\varnothing$ for any $r<t<2r.$ Indeed,
$a_0^{\,\prime}\in S(a^{\,*}, r)\cap f(A),$ therefore $S(a^{\,*},
r)\cap f(A)\ne\varnothing.$ Besides that, since $A\cap ({\Bbb
R}^n\setminus D_0)\ne \varnothing,$ there exists $b\in A\cap ({\Bbb
R}^n\setminus D_0).$ Since $f(x)=x$ for $x\in D\setminus D_0,$
$f(b)=b\in f(A).$ By the definition of $r$ in~(\ref{eq2}) $A\cap
({\Bbb R}^n\setminus B(a^{\,*}, 2r))\ne\varnothing.$ Now, due to
\cite[Theorem~1.I.5.46]{Ku} we obtain that $S(a^{\,*}, 2r)\cap
f(A)\ne\varnothing.$ But now, by the connectedness of $f(A),$ we
obtain that $S(a^{\,*}, t)\cap f(A)\ne\varnothing$ for all $r<t<2r$
(\cite[Theorem~1.I.5.46]{Ku}). Arguing similarly, by the definition
of $r$ in~(\ref{eq2}), we obtain that $2r\leqslant \max\limits_{x\in
A^{\,*}}d(x, a^{\,*}).$ Besides that, $B(a^{\,*}, r)\cap A^{\,*}\ne
\varnothing,$ because $a^{\,*}\in A^{\,*}.$ Now, due to
\cite[Theorem~1.I.5.46]{Ku} we obtain that $S(a^{\,*}, t)\cap
A^{\,*}\ne\varnothing$ for any $r<t<2r.$

\medskip
In this case, by Proposition~\ref{pr1} we obtain that
\begin{equation}\label{eq3}
M_p(\Gamma(f(A), f(A^{\,*}), D))\geqslant
\frac{2^nb_{n,p}}{n-p}((2r)^{n-p}-r^{n-p})\,.
\end{equation}
On the other hand, by the definition of~$f$ and by Lemma~\ref{lem2}
$$M_p(\Gamma(f(A), f(A^{\,*}), D))=$$
\begin{equation}\label{eq4}
=M_p(f(\Gamma(A), f(A^{\,*}), D))\leqslant K_{n, p}\cdot
M_p(\Gamma(A, A^{\,*}, D)\,.
\end{equation}
Uniting the relations~(\ref{eq3}) and~(\ref{eq4}) we conclude that
$$M_p(\Gamma(A, A^{\,*}, D)\geqslant \frac{1}{K_{n, p}}
\frac{2^nb_{n,p}}{n-p}((2r)^{n-p}-r^{n-p})>0\,.$$
Theorem is proved.~$\Box$

\section{Proof of Theorem~\ref{th3}}
We will carry out the proof using the approach
from~\cite[Theorem~1.16]{Na$_2$}, see also lemma~3.8.1
in~\cite{Sev$_1$}. Since the proof for the case $p=n$ is given
in~\cite[Lemma~3.8.1]{Sev$_1$}, it suffices to consider the case
$p\ne n.$ Let $U$ be an arbitrary neighborhood of the point $x_0$
and $E^{\,*}$ be a continuum in $D.$ By the condition of the Lemma,
there exists a neighborhood $V\subset D\cap U,$ a compact set
$E\subset D$ and a number $\delta>0$ such that the
relation~(\ref{eq17***}) is fulfilled. Without loss of generality,
we may assume that
\begin{equation}\label{eq21***}
E\cap E^{\,*}= \varnothing\,.
\end{equation}
Indeed, in the contrary case we may consider any continuum
$E^{\,*}_1\subset D\setminus (E\cup E^{\,*})$ and to prove this
statement firstly for $E$ and $E^{\,*}_1,$ and after that for
$E^{\,*}_1$ and $E^{\,*}.$ We may consider that $V\subset B(x_0,
t_0),$ where $t_0=\frac{1}{2}\cdot{\rm dist\,}(x_0, E).$ Let $F$ be
continuum in $D$ such that $F\cap\partial U\ne\varnothing \ne F\cap
\partial V.$ Due to~(\ref{eq21***}), we obtain that
$$4r: =\min\left\{{\rm dist\,}(E, E^{\,*}), \quad {\rm dist\,}(E, \partial D),
\quad {\rm dist\,} (E, B(x_0, t_0))\right\} > 0\,.$$
Let $E_1,\ldots, E_q$ be a finite covering of $E$ byu closed balls
centered at the points $e_i\in E_i$ and of the radius $r,$
$i=1,\ldots, q.$ Recall that, $p$-modulus of families of paths
joining two continua in $D$ is positive for $p>n-1$ (see
Theorem~\ref{th2}). Now, we set
\begin{equation}\label{eq22***}
\Gamma_i^{\,*}=\Gamma\left(E_i, E^{\,*}, D\right)\,,\qquad
M_p(\Gamma_i^{\,*})=\delta_i>0\,,
\end{equation}
$$\delta^{\,*}= 3^{\,-p}\min\left\{\delta/q, \delta_1,\ldots,
\delta_q,\right.$$
\begin{equation}\label{eq24***}
\left.\frac{2^nb_{n,p}}{n-p}((2r)^{n-p}-r^{n-p}), \\
\frac{2^nb_{n,p}}{n-p}((4r)^{n-p}-(2r)^{n-p})\right\}\,,\end{equation}
where $b_{n,p}$ is a constant from Proposition~\ref{pr1}, and
$$\Gamma=\Gamma(E, F, D)\,,\qquad \Gamma_i=\Gamma(E_i, F, D)\,,
\qquad \Gamma^{\,*}=\Gamma\left(E^{\,*}, F, D\right)\,.$$ Due to the
subadditivity of the modulus of families of paths,
\begin{equation}\label{eq26***}
0<\delta \leqslant M_p(\Gamma)\leqslant
M_p\left(\Gamma\left(\bigcup\limits_{i=1}^q E_i, F,
D\right)\right)\leqslant \sum\limits_{i=1}^q M_p(\Gamma_i)\,.
\end{equation}
By~(\ref{eq26***}) it follows that $M_p(\Gamma_{i_0})\geqslant
\delta/q$ at least for some $i_0\in \{1,\ldots,q\}.$ Now, we may
assume that
\begin{equation}\label{eq27***}
M_p(\Gamma_1)\ \geqslant \delta/q\,.
\end{equation}
Let us show that
\begin{equation}\label{eq32***}
M_p(\Gamma^{\,*})\geqslant \delta^{\,*}\,.
\end{equation}
By the definition, the modulus of families of paths, containing at
least one constant path, equals to infinity. Due to this, it is
sufficiently to consider that $E^{\,*}\cap F\ne \varnothing.$ Put
$\rho\in {\rm adm\,}\Gamma^{\,*}.$ If at least one of the conditions
\begin{equation}\label{eq28***}
\int\limits_{\gamma_1}\rho \,|dx|\geqslant 1/3\,, \qquad 
\int\limits_{\gamma_1^{\,*}}\rho\, |dx|\geqslant 1/3\,,
\end{equation}
hold for $\gamma_1\in \Gamma_1,$ $\gamma_1^{\,*}\in \Gamma_1^{\,*},$
then we obtain that $3\rho\in {\rm adm\,}\Gamma_1$ and $3\rho\in
{\rm adm\,}\Gamma_1^{\,*}.$ Now, we obtain that
\begin{equation}\label{eq31***}
\int\limits_{D} \rho^p(x)\,dm(x)\geqslant
3^{\,-p}\min\{M_p(\Gamma_1), M_p(\Gamma_1^*)\}\,.
\end{equation}
The relation~(\ref{eq31***}) together with (\ref{eq24***}),
(\ref{eq26***}) and (\ref{eq27***}) gives (\ref{eq32***}).

Assume now that, there are at least two paths $\gamma_1\in \Gamma_1$
and $\gamma_1^*\in \Gamma_1^*$ for which the relations
in~(\ref{eq28***}) do not hold. Assume firstly that~$F\cap B(e_1,
2r)=\varnothing.$ Let
$$R_1:=\{x\in {\Bbb R}^n: r<|x-e_1|<2r\}\,,\qquad
\Delta_1:=\Gamma\left(|\gamma_1|, |\gamma_1^{*}|, R_1\right)\,.$$
Since $\rho\in {\rm adm\,}\Gamma^{\,*},$ by the definition of
$\gamma_1$ and $\gamma_1^*,$
\begin{equation}\label{eq33***}
\int\limits_{\alpha_1}\rho \,|dx|\geqslant 1/3
\end{equation}
for any path $\alpha_1\in \Delta_1.$ Observe that, $|\gamma_1|\cap
S(e_1, t)\ne\varnothing\ne|\gamma_1^{\,*}|\cap S(e_1, t)$ for any
$t\in (r, 2r),$ in addition, $B(e_1, 2r)\subset D.$ Thus, by
Proposition~\ref{pr1}, we obtain that
\begin{equation}\label{eq25***}
\int\limits_{D} \rho^p(x)\,dm(x)\geqslant
3^{-p}\frac{2^nb_{n,p}}{n-p}((2r)^{n-p}-r^{n-p})\,.
\end{equation}
However, the relation~(\ref{eq25***}) together with~(\ref{eq22***})
proves~(\ref{eq32***}). Now, let us assume that $F\cap B(e_1, 2r)\ne
\varnothing.$ Let
$$R_1^{\,*}:=\{x\in {\Bbb R}^n: 2r<|x-e_1|<4r\}\,,\qquad
\Delta_1^{\,*}:=\Gamma\left(F, |\gamma_1^{*}|, R_1^{\,*}\right)\,.$$
Observe that, in this case,
\begin{equation}\label{eq34***}
\int\limits_{\alpha_1^{\,*}}\rho \,ds\geqslant 1/3
\end{equation}
for any path $\alpha_1^{\,*}\in \Delta_1^{\,*}.$ Besides that, by
the definition of $r$ and due to $V\subset B(x_0, t_0),$ $t_0={\rm
dist\,}(x_0, E) /2,$ and $F\cap\partial U\ne\varnothing \ne
F\cap\partial V,$ we obtain that $4r\leqslant
|x_1-\widetilde{x_1}|,$ where $\widetilde{x_1}\in F\cap B(x_0, t).$
Thus, by~\cite[Theorem~1.I.5.46]{Ku} the continuum $F$ intersects
$S(e_1, t)$ for any $t\in (2r, 4r.)$ By Proposition~\ref{pr1}
\begin{equation}\label{eq35***}
\int\limits_{D} \rho^p(x)\,dm(x)\geqslant
3^{-p}\frac{2^nb_{n,p}}{n-p}((4r)^{n-p}-(2r)^{n-p})\,.
\end{equation}
Theorem is proved.~$\Box$

\section{Proof of Theorem~\ref{th5}} Let $x_0\in \partial
D\setminus\{\infty\}.$ Let us apply the approach used under the
proof of~\cite[Theorem~6.4]{Na$_1$}. We will carry out the proof by
the method of the opposite. Assume that, the statement of the
theorem is not true. Then there is $r_0>0$ such that the set $U\cap
D$ has infinitely many components for any neighborhood $U$ of the
point $x_0,$ $U\subset B(x_0, r_0).$ Chose $0<r<r_0$ and a
neighborhood $U$ such that $\overline{U}\subset B(x_0, r).$ Due to
the strong accessibility of $\partial D$ at the point $x_0$ with
respect to $p$-modulus, there exists a neighborhood $V\subset U$ of
a point $x_0$ and a compact set $F\subset D$ and a number $P>0$ such
that
\begin{equation}\label{eq2A}
M_p(\Gamma(E, F, D))\geqslant P
\end{equation}
fir any continuum $E$ in $D$ such that $E\cap
\partial U\ne\varnothing\ne E\cap
\partial V.$ By Theorem~\ref{th3} we may assume that $F$ is outside of
$B(x_0, r_0).$ We also may assume that $\overline{V}\subset U.$

\medskip
Observe that, there are infinitely many components $\widetilde{E_k}$
of the set $B(x_0, r_0)\cap D$ such that $\widetilde{E_k}\cap
\partial V\ne \varnothing,$ $k=1,2,\ldots.$ Indeed, in the contrary case
$(B(x_0, r_0)\cap D)\cap V=\widetilde{E_{i_1}}\cup\ldots
\widetilde{E_{i_{k_0}}},$ $1\leqslant k_0<\infty,$ that contradicts
to the definition of $r_0.$

\medskip
Observe that, the sequence $\widetilde{E_k}$ contains a subsequence
$E_l:=\widetilde{E_{k_l}},$ $l=1,2,\ldots, $ such that
$m(\widetilde{E_{k_l}})<\frac{1}{l},$ because in the contrary case
there exists $l_0>0$ such that $m(\widetilde{E_{k}})>l_0$ for any
natural $k.$ In this case,
$$\Omega_n2^nr_0^n=m(B(x_0, r_0))\geqslant \sum\limits_{k=1}^{\infty}
m(\widetilde{E_{k}})\geqslant
\sum\limits_{k=1}^{\infty}l_0=\infty\,,$$
which is impossible.

\medskip
Below we use the index $k$ for the notation of $E_k,$ $k=1,2,\ldots
.$ Fix $x_k\in V\cap E_k$ and $y_k\in D\setminus E_k$ and let us to
consider a path $\gamma_k,$ $\gamma_k:[0, 1]\rightarrow D$ joining
$x_k$ and $y_k$ in $D,$ i.e., $\gamma_k(0)=x_k$ and
$\gamma_k(1)=y_k.$ By~\cite[Theorem~1.I.5.46]{Ku} there exists
$0<t_1<1$ such that $z_k:=\gamma_k(t_1)\in \partial E_k.$ Observe
that, $z_k\in S(x_0, r_0)\cap D$ (see, e.g., \cite[Theorem~1,
$\S\,47,$ item~III]{Ku}). Without loss of generality, we may assume
that $\gamma_k(t)\in E_k$ for any $0\leqslant t<t_1.$ Since
$U\subset \overline{B(x_0, r)}$ and $\overline{V}\subset U,$ again
due to \cite[Theorem~1.I.5.46]{Ku} there are $0<t_3<t_2<t_1$ such
that $\gamma_k(t_3)\in \partial V$ and $\gamma_k(t_2)\in \partial
U.$

\medskip
Let $F_k:=|\gamma_{k}|_{[t_3, t_2}|.$ By the construction, $F_k$
intersects $\partial U$ and $\partial V$ for any $k\in {\Bbb N}.$
Let $\Gamma_k$ be a family of paths joining $E_k$ and $F$ in $D.$
Since by the choosing $F\subset D\setminus B(x_0, r)$ and
$\overline{U}\subset B(x_0, r),$ the length of any path $\gamma\in
\Gamma_k$ is at least $${\rm dist\,}(\partial U, S(x_0,
r_0))\geqslant {\rm dist\,}(S(x_0, r), S(x_0, r_0))=r_0-r\,.$$

\medskip
Set
$$\rho(x):=\begin{cases}\frac{1}{r_0-r}\,,& x\in E_k\,,\\
0\,,& x\not\in E_k \end{cases}\,.$$
Then $\rho\in {\rm adm\,}\Gamma_k.$ Thus
\begin{equation}\label{eq12A}
M_p(\Gamma_k)\leqslant\int\limits_{E_k}\frac{dm(x)}{r^p}=\frac{1}{kr^p}
\rightarrow 0\,,\qquad k\rightarrow\infty.
\end{equation}
On the other hand, by~(\ref{eq2A})
\begin{equation}\label{eq3A}
M_p(\Gamma_k)=M_p(\Gamma(F_k, F, D))\geqslant P\,.
\end{equation}
The relations~(\ref{eq12A}) and~(\ref{eq3A}) contradict each other.
The contradiction obtained above point to finitely connectedness of
$D$ at $x_0.$ Theorem is proved.~$\Box$

CONTACT INFORMATION

\medskip
\medskip
\noindent{{\bf Evgeny Oleksandrovych Sevost'yanov} \\
{\bf 1.} Zhytomyr Ivan Franko State University,  \\
40 Velyka Berdychivs'ka Str., 10 008  Zhytomyr, UKRAINE \\
{\bf 2.} Institute of Applied Mathematics and Mechanics\\
of NAS of Ukraine, \\
19 Henerala Batiuka Str., 84 100 Slovians'k,  UKRAINE\\
esevostyanov2009@gmail.com}

\medskip
\noindent{{\bf Zarina Olimzhonivna Kovba} \\
Zhytomyr Ivan Franko State University,  \\
40 Velyka Berdychivs'ka Str., 10 008  Zhytomyr, UKRAINE \\
e-mail: mazhydova@gmail.com  }

\medskip
\noindent{{\bf Heorhii Olegovych Nosal} \\
Zhytomyr Ivan Franko State University,  \\
40 Velyka Berdychivs'ka Str., 10 008  Zhytomyr, UKRAINE \\
e-mail: nocalko@gmail.com }

\medskip
\noindent{{\bf \noindent Nataliya Ilkevych} \\
Zhytomyr Ivan Franko State University,  \\
40 Velyka Berdychivska Str., 10 008  Zhytomyr, UKRAINE \\
ilkevych1980@gmail.com}


\begin{thebibliography}{99}

\bibitem{Car94}
Caraman~P., \emph{Relations between $p$-capacity and $p$-module
(I)}, Rev. Roum. Math. Pures Appl., \textbf{39}, no.~6, 1994,
p.~509--553.

\bibitem{Ku} Kuratowski K., \emph{Topology, v. 2.} -- New York--London: Academic Press, 1968.

\bibitem{MRSY} Martio O., Ryazanov V., Srebro U. and Yakubov
E., \emph{Moduli in Modern Mapping Theory.} -- New York: Springer
Science + Business Media, LLC, 2009.

\bibitem{Na$_2$} N\"{a}kki~R., \emph{Boundary behavior
of quasiconformal mappings in $n$-space}, Ann. Acad. Sci. Fenn. Ser.
A., \textbf{484,} 1970, p.~1--50.

\bibitem{Na$_1$} N\"akki~R., \emph{Extension of Loewner's capacity
theorem}, Trans. Amer. Math. Soc., \textbf{180,} 1973, p.~229--236.

\bibitem{NP} N\"{a}kki~R. and Palka~B., \emph{Uniform equicontinuity
of  quasiconformal mappings}, Proc. Amer. Math. Soc., \textbf{37},
no.~2, 1973, p.~427--433.

\bibitem{Sal} Salimov R.R.,
\emph{On a new condition of finite Lipschitz of Orlicz-Sobolev
class,} Matematychni Studii, \textbf{44}, no.~1, 2015, p.~27--35 (in
Russian).

\bibitem{SalSev} Salimov~R.R. and Sevost'yanov~E.A.,
\emph{The Poletskii and V\"{a}is\"{a}l\"{a} inequalities for the
mappings with $(p, q)$-distortion}, Complex Variables and Elliptic
Equations, \textbf{59}, no.~2, 2014, p.~217--231.

\bibitem{Sev$_1$} Sevost'yanov~E.A.,
\emph{Study of Space Mappings by a Geometric Method.} -- Kyiv:
Naukova Dumka, 2014 (in Russian).

\bibitem{Va} V\"{a}is\"{a}l\"{a} J., \emph{Lectures on $n$-Dimensional
Quasiconformal Mappings}, Lecture Notes in Math. \textbf{229},
Berlin etc.: Springer--Verlag, 1971.


\end{thebibliography}
\end{document}